\documentclass[12pt]{article}
\usepackage[utf8]{inputenc}
\usepackage{amsmath,amssymb,amsthm}
\usepackage[margin=1in]{geometry}
\usepackage[hidelinks]{hyperref}

\newenvironment{romenum}
  {\begin{enumerate}}
  {\end{enumerate}}

\theoremstyle{plain}
\newtheorem{theorem}{Theorem}
\newtheorem{lemma}[theorem]{Lemma}
\newtheorem{proposition}[theorem]{Proposition}

\theoremstyle{definition}
\newtheorem{definition}[theorem]{Definition}
\theoremstyle{remark}
\newtheorem{remark}[theorem]{Remark}

\newenvironment{salign}{
    \begin{equation}
    \begin{aligned}
}{
    \end{aligned}
    \end{equation}
    \ignorespacesafterend
}

\newcommand{\defeq}{\stackrel{\textnormal{def}}{=}}

\newcommand{\Z}{\mathbb{Z}}
\newcommand{\Q}{\mathbb{Q}}

\newcommand{\F}{\mathbb{F}}
\newcommand{\Norm}{\mathrm{N}}
\newcommand{\OO}{\mathcal{O}}

\newcommand{\Invisible}{\mathrm{Invisible}}

\title{Lattice Point Visibility along Powers of Quadratic Polynomials}

\author{
  Abraham Lobsenz \\
  \and
  Tristan Phillips \\
}

\date{}
\hypersetup{
  pdftitle={Lattice Point Visibility along Powers of Quadratic Polynomials},
  pdfauthor={Abraham Lobsenz and Tristan Phillips}
}

\begin{document}
\maketitle

\begin{abstract}
We study the growth of the number of invisible lattice points along powers of quadratic polynomials. Let
$f(x)=Ax^2+Bx+C\in\Z[x]$ have positive leading coefficient and nonzero
discriminant, and let $F(x)=f(x)^m$ with $m\geq 2$. For $m\geq 3$ we prove that the number
of invisible lattice points in $[1,N]^2$ has order $N\log N$, and when $m=2$ the number of invisible lattice points satisfies
$N\log N\ll_F\#\mathrm{Invisible}_F(N)\ll_F N(\log N)^4$. These estimates refine a previous result of the authors. 
\end{abstract}

\section{Introduction}\label{sec:intro}

Lattice point visibility is a classical problem in number theory.
A point $(a,b)$ in the integer lattice $\Z\times \Z$ is said to be \emph{visible from the origin} if the line segment between the points $(0,0)$ and $(a,b)$ contains no additional lattice points. The lattice point visibility problem asks: What proportion of lattice points are visible from the origin?

Since $(a,b)$ being visible from the origin is equivalent to $\gcd(a,b)=1$, the lattice point visibility problem is equivalent to asking for the probability that two randomly chosen integers are coprime. This reformulation relates the problem to the Basel problem, posed by Mengoli in 1650, which asks for the precise value of the sum
\begin{equation}
    \zeta(2) = \sum_{n=1}^\infty \frac{1}{n^2}.
\end{equation}
Euler's celebrated 1734 solution of this problem showed that
\begin{equation}
    \zeta(2) = \prod_{p \text{ prime}} \left( 1 - \frac{1}{p^2}\right)^{-1} =\frac{\pi^2}{6}.
\end{equation}
Since the probability that two integers are not both divisible by a given prime $p$ is $1-1/p^2$, this shows that the probability that two integers are coprime is $6/\pi^2$, and hence the density of visible lattice points is also $6/\pi^2$. Subsequent work justifying and extending Euler's argument is due to Weierstrass, Mertens, Dirichlet, Ces\`{a}ro, and Sylvester.
For more modern treatments of the classical lattice point visibility problem, see \cite[Theorem 332]{HardyWright} and \cite[\S 3.2]{Apostol}.

In our earlier paper \cite{LobsenzPhillipsPila}, we studied lattice point visibility along polynomial lines of sight; we recall
the definition of visibility and invisibility.

\begin{definition}[{\cite[Definition~2.1]{LobsenzPhillipsPila}}]\label{def:visibility}
Let $F(x)\in\Z[x]$ be a polynomial with positive leading coefficient. A lattice point
$(a,h)\in\Z_{>0}\times\Z_{>0}$ is \emph{visible along $F(x)$} if there
exists a rational number $t\in\Q$ such that $h=tF(a)$ and $a$ is the
smallest positive integer $u$ such that $tF(u)$ is a positive integer.
A lattice point is \emph{invisible along $F(x)$} if it is not visible.
\end{definition}

For $F$ satisfying $F(0)=0$, Definition~\ref{def:visibility} recovers the polynomial visibility studied
by Goins--Harris--Kubik--Mbirika~\cite{GoinsHarrisKubikMbirika} and
Chaubey--Pandey~\cite{ChaubeyPandey}, who restricted to polynomials passing
through the origin.

For $N\in\Z_{>0}$ define $[1,N]\defeq\{1,\dots,N\}$. Following \cite[(12)--(13)]{LobsenzPhillipsPila}, define the sets of
visible and invisible lattice points in $[1,N]^2$:
\begin{equation}\label{eq:def-visible-invisible}
\begin{aligned}
    \mathrm{Visible}_F(N)
        &\defeq \{(a,h)\in[1,N]^2:(a,h)\text{ is visible along }F\},\\
    \mathrm{Invisible}_F(N)
        &\defeq \{(a,h)\in[1,N]^2:(a,h)\text{ is invisible along }F\}.
\end{aligned}
\end{equation}
Write $\#\mathrm{Visible}_F(N)$ and $\#\mathrm{Invisible}_F(N)$ for their
sizes; since these sets partition $[1,N]^2$,
\begin{equation}\label{eq:partition}
    \#\mathrm{Visible}_F(N)+\#\mathrm{Invisible}_F(N)=N^2.
\end{equation}
The \emph{visibility density} of $F$ is
\begin{equation}\label{eq:DF}
    D(F)\defeq \lim_{N\to\infty}\frac{\#\mathrm{Visible}_F(N)}{N^2}
            = 1-\lim_{N\to\infty}\frac{\#\mathrm{Invisible}_F(N)}{N^2},
\end{equation}
whenever this limit exists.

For the classical problem one has $D(x)=1/\zeta(2)=6/\pi^2$, and
\cite{GoinsHarrisKubikMbirika} extends this to monomial lines of sight, showing $D(\alpha x^b)=1/\zeta(1+b)$. The behavior changes dramatically as
soon as $F$ has more than one distinct root. In
\cite[Conjecture~1.1]{LobsenzPhillipsPila} we conjectured that $D(F)=1$
for every $F\in\Z[x]$ with at least two distinct roots, extending a
conjecture of Chaubey and Pandey \cite[Conjecture~1.6]{ChaubeyPandey}.
The conjecture of Chaubey and Pandey concerned the case in which $F(0)=0$, and was resolved by Chaubey, Pandey, and Regavim \cite{ChaubeyPandeyRegavim}.
More precisely, in the case $F(0)=0$ they show that for any $\varepsilon>0$,
\begin{equation}
    \#\Invisible_F(N)\ll_{F,\varepsilon}
    \begin{cases}
        N^{2-1/(2\deg(F)-1)+\varepsilon} &\text{ if $F$ is separable,}\\
        N^{2-1/(\deg(F)^2-1)+\varepsilon} &\text{ in general.}
    \end{cases}
\end{equation}

A result of the authors, \cite[Theorem~1.2]{LobsenzPhillipsPila}, verifies that $D(F)=1$ for
$F(x)=f(x)^m$ whenever $f\in\Z[x]$ has degree at least $2$, positive leading coefficient, and at least two distinct roots over $\overline{\Q}$,
and $m\geq 2$. More precisely, in \cite[Theorem~1.3]{LobsenzPhillipsPila} it is shown that there exists an integer $2\leq \delta_F\leq \deg(F)$ depending on $F$, such that for any $\varepsilon>0$,
\begin{equation}\label{eq:pila-bound}
    \#\mathrm{Invisible}_F(N)
    \ll_{F,\varepsilon}N^{1+1/\delta_F+\varepsilon}.
\end{equation}
This bound was obtained by applying Pila's bound for integral points on algebraic curves \cite{Pila} to a certain family of auxiliary curves.

The purpose of this paper is to sharpen the upper bounds for
$\#\mathrm{Invisible}_F(N)$ when $f$ is quadratic. We also establish
complementary lower bounds. In this case the auxiliary curves are conics, and
counting their integer points reduces to counting solutions to a generalized
Pell equation.

Throughout the remainder of the paper, let $f(x)\in\Z[x]$ denote the
quadratic polynomial
\begin{equation}\label{eq:f-quadratic}
    f(x)\defeq Ax^2+Bx+C,
\end{equation}
satisfying $A>0$ and discriminant
\begin{equation}\label{eq:Delta}
    \Delta\defeq B^2-4AC\neq 0.
\end{equation}
For a fixed integer $m\geq 2$, let $F$ denote the $m$th power of $f$,
\begin{equation}\label{eq:F}
    F(x)\defeq f(x)^m.
\end{equation}

Our main result sharpens the bound \eqref{eq:pila-bound} in this setting.

\begin{theorem}\label{thm:visibility}
Let $F$ be as in equation \eqref{eq:F}.
\begin{romenum}
\item If $m\geq 3$, then
\begin{equation}
    \#\mathrm{Invisible}_F(N)\asymp_F N\log N.
\end{equation}
\item If $m=2$, then
\begin{equation}
    N\log N\ll_F \#\mathrm{Invisible}_F(N)\ll_F N(\log N)^4.
\end{equation}
\end{romenum}
\end{theorem}

Theorem~\ref{thm:visibility} is deduced from bounds on a certain
GCD sum, which we now describe. Since $A>0$, the polynomial $f$ is
eventually positive and strictly increasing and $f'(x)=2Ax+B$ is eventually positive; as in
\cite[\S2]{LobsenzPhillipsPila}, we choose an integer $n_f\geq 1$ such
that whenever $a>b>n_f$,
\begin{equation}\label{eq:monotone}
    0<f(b)<f(a)\qquad\text{and}\qquad 2Ab+B>0.
\end{equation}
The reduction at the heart of \cite{LobsenzPhillipsPila} expresses the
invisibility count in terms of the sum
\begin{equation}\label{eq:S_F}
    S_F(N)\defeq
    \sum_{1\leq a\leq N}\sum_{n_f<b<a}
    \frac{\gcd(F(a),F(b))}{F(a)},
\end{equation}
which we will refer to as a \emph{GCD sum}.
By \cite[Proposition~2.4]{LobsenzPhillipsPila},
\begin{equation}\label{eq:invisible-S-intro}
    \#\mathrm{Invisible}_F(N)\leq N\,S_F(N)+O_F(N).
\end{equation}
Thus any estimate of the form $S_F(N)=o(N)$ forces
$\#\mathrm{Invisible}_F(N)=o(N^2)$, and hence $D(F)=1$ by \eqref{eq:DF}.
The following theorem gives such an estimate with explicit savings and,
via \eqref{eq:invisible-S-intro}, immediately implies the upper bounds.

\begin{theorem}\label{thm:upper}
Let $f$, $m$, and $F$ be as in Theorem~\ref{thm:visibility}.
\begin{romenum}
\item If $m\geq 3$, then
    \begin{equation}
        S_F(N)\ll_{F}\log N.
    \end{equation}
\item If $m=2$, then
    \begin{equation}
        S_F(N)\ll_F(\log N)^4.
    \end{equation}
\end{romenum}
\end{theorem}

Our lower bounds apply, more generally, to all powers $m\geq 1$.

\begin{theorem}\label{thm:lower}
Let $f$ be as in equation \eqref{eq:f-quadratic}, and let $F=f^m$ with $m\geq 1$.
Then
\begin{equation}
    S_F(N)\gg_F \log N
    \qquad
    \text{and}
    \qquad
    \#\mathrm{Invisible}_F(N)\gg_F N\log N.
\end{equation}
In particular, if $m\geq 3$, then
\begin{equation}
    S_F(N)\asymp_F \log N.
\end{equation}
\end{theorem}

\section*{Acknowledgements}
We would like to thank Sneha Chaubey, Ashish Kumar Pandey, and Shvo Regavim for sharing a preprint of their paper \cite{ChaubeyPandeyRegavim} with us.
We thank Dartmouth College for access to various AI tools, which we used in the preparation of this manuscript.
TP was supported by the National Science Foundation through grant DMS-2303011.

\section{Specializing the GCD sum to quadratic \texorpdfstring{$f$}{f}}\label{sec:gcd-sum}

The inequality \eqref{eq:invisible-S-intro} bounds
$\#\mathrm{Invisible}_F(N)$ in terms of the GCD sum $S_F(N)$ for any
$F\in\Z[x]$ with positive leading coefficient. In this section we
specialize to $F=f^m$, with $f$ the quadratic polynomial in
\eqref{eq:f-quadratic}, and rewrite $S_F(N)$ as a weighted count of
integer points on a family of conics. This sets up the Pell-equation
analysis carried out in \S\ref{sec:pell}.

For $a>b>n_f$, define
\begin{equation}\label{eq:g-s-r}
    g(a,b)\defeq \gcd(f(a),f(b)), \qquad
    s(a,b)\defeq \frac{f(a)}{g(a,b)}, \qquad
    r(a,b)\defeq \frac{f(b)}{g(a,b)}.
\end{equation}
Then $\gcd(s(a,b),r(a,b))=1$ and $s(a,b)>r(a,b)\geq 1$. The definitions in this section and the identity \eqref{eq:S-rearranged} are
valid for every $m\geq 1$. Moreover,
\begin{salign}
    \frac{\gcd(F(a),F(b))}{F(a)}
    &= \frac{\gcd(f(a)^m,f(b)^m)}{f(a)^m} \\
    &= \frac{g(a,b)^m}{f(a)^m} \\
    &= \frac{1}{s(a,b)^m}.
\end{salign}

For positive integers $s>r$, define
\begin{equation}\label{eq:Msr}
    M_{s,r}(N)\defeq
    \#\{(a,b)\in[n_f+1,N]^2 :
       b<a,\; s(a,b)=s,\; r(a,b)=r\}.
\end{equation}
Since $s(a,b)\leq f(a)\leq f(N)$ for $a\leq N$, we obtain
\begin{equation}\label{eq:S-rearranged}
    S_F(N)
    =
    \sum_{2\leq s\leq f(N)}\sum_{r=1}^{s-1}
    \frac{M_{s,r}(N)}{s^m}.
\end{equation}

The point of this rearrangement is that the pairs counted by
$M_{s,r}(N)$ lie on a fixed algebraic curve.

\begin{proposition}\label{prop:curve}
If $(a,b)$ is counted by $M_{s,r}(N)$, then
\begin{equation}
    s\,f(b)=r\,f(a).
\end{equation}
Equivalently, $(a,b)$ is an integer point on the curve
\begin{equation}\label{eq:curve}
    G_{s,r}(x,y)\defeq s\,f(y)-r\,f(x)=0.
\end{equation}
\end{proposition}

\begin{proof}
If $(a,b)$ is counted by $M_{s,r}(N)$, then by definition
\begin{equation}
    f(a)=s\,g(a,b),\qquad f(b)=r\,g(a,b).
\end{equation}
Multiplying the second equality by $s$ and the first by $r$ gives
$s\,f(b)=r\,f(a)$.
\end{proof}

\section{Reduction to a Pell equation and ideal counting}\label{sec:pell}

In this section we transform the equation $s\,f(b)=r\,f(a)$ from
Proposition~\ref{prop:curve} into a generalized Pell equation, and then
bound the number of integer solutions by combining a Pell-family count
with an ideal-counting argument. These bounds will be assembled in
\S\ref{sec:upper-proof} to prove Theorem~\ref{thm:upper}.

Throughout this section, $f(x)=Ax^2+Bx+C$ satisfies the hypotheses of
Theorem~\ref{thm:visibility}.

\subsection{Reduction to a generalized Pell equation}

Fix positive integers $s>r$. The curve equation $s\,f(b)=r\,f(a)$ is
\begin{equation}
    s(Ab^2+Bb+C)=r(Aa^2+Ba+C).
\end{equation}
Using
\begin{equation}
    4Af(x)=(2Ax+B)^2-\Delta,
\end{equation}
we obtain
\begin{equation}
    r(2Aa+B)^2-s(2Ab+B)^2=(r-s)\Delta.
\end{equation}
Define
\begin{equation}\label{eq:pell-change}
    X\defeq r(2Aa+B), \qquad
    Y\defeq 2Ab+B, \qquad
    D\defeq sr, \qquad
    n\defeq r(r-s)\Delta.
\end{equation}
Since $s>r\geq 1$ and $\Delta\neq 0$, we have $D\geq 2$ and $n\neq 0$.
Then
\begin{equation}\label{eq:pell-main}
    X^2-DY^2=n.
\end{equation}
For fixed $s$ and $r$, the map $(a,b)\mapsto(X,Y)$ is injective.
Moreover, if $(a,b)$ is counted by $M_{s,r}(N)$, then
\begin{equation}\label{eq:Y-bound}
    |Y|\leq 2AN+|B|\ll_f N.
\end{equation}
Thus $M_{s,r}(N)$ is bounded by the number of integer solutions of the generalized Pell equation
\eqref{eq:pell-main} satisfying the bound \eqref{eq:Y-bound}. 
 \subsection{Pell families}

For a nonsquare integer $D\geq 2$, we write $K\defeq\Q(\sqrt D)$ for the
associated real quadratic field, $\OO_K$ for its ring of integers,
$\Norm$ for the field norm $K\to\Q$, and $R\defeq\Z[\sqrt D]$ for the order
generated by $\sqrt D$.

\begin{definition}[Pell family]\label{def:pell-family}
Let $D\geq 2$ be a nonsquare integer and let $n\neq 0$. A
\emph{Pell family} of integer solutions to
\begin{equation}\label{eq:pell}
    X^2-DY^2=n
\end{equation}
is an equivalence class of solutions $(X,Y)\in\Z^2$ under the relation
\begin{equation}
    (X,Y)\sim (X',Y')
    \iff
    X'+Y'\sqrt D
    =
    (X+Y\sqrt D)\,u
\end{equation}
for some unit $u\in R^\times$.
We write $\mathcal{F}(D,n)$ for the number of Pell families.
\end{definition}

For $n\in \Z_{>0}$ a positive integer, let $\tau(n)$ denote the divisor function
\begin{equation}
    \tau(n)\defeq \sum_{\substack{d\in \Z_{>0} \\ d|n}} 1.
\end{equation}

\begin{lemma}[Representation bound]\label{lem:rep-bound}
Let $D\geq 2$ be squarefree. For every nonzero integer $n$,
\begin{equation}\label{eq:divisor_bound}
    \mathcal{F}(D,n)\leq 5\,\tau(|n|),
\end{equation}
where $\tau$ denotes the divisor function. In particular, for every
$\varepsilon>0$,
\begin{equation}\label{eq:divisor_bound_epsilon}
    \mathcal{F}(D,n)\ll_{\varepsilon}|n|^\varepsilon,
\end{equation}
uniformly in $D$.
\end{lemma}

\begin{proof}
A solution $(X,Y)$ to the generalized Pell equation \eqref{eq:pell} gives an element
\begin{equation}
    \alpha=X+Y\sqrt D\in R
\end{equation}
with $\Norm(\alpha)=n$. 
Solutions in the same Pell family as $(X,Y)$ correspond to elements of the form $\alpha u\in R$ with $u\in R^\times$ a unit.
Note that the elements $\alpha$ and $\alpha u$ generate the same principal ideal, and conversely any two generators of the same principal ideal differ by a unit. 
Hence distinct Pell families correspond to distinct principal ideals $\alpha R$ with index $[R:\alpha R]=|n|$. Then, letting
\begin{equation}
    a_R(m)\defeq \#\{\alpha R\subset R : [R : \alpha R] = m\}
\end{equation}
denote the number of principal ideals of $R$ with norm $m$, 
we have
\begin{equation}\label{eq:F-ideal}
    \mathcal{F}(D,n)\leq a_R(|n|).
\end{equation}
Let 
\begin{equation}
    a_K(m)\defeq
    \# \{J\subseteq \OO_K : J \text{ is a nonzero ideal and }[\OO_K:J]=m\}
\end{equation}
denote the number of integral $\OO_K$-ideals of norm $m$.
Since ideals in the Dedekind domain $\OO_K$ factor uniquely into prime ideals,
the number of ideals of norm $p^k$ is at most $k+1$ for every rational prime
$p$ and every $k\geq 0$; this maximum occurs when $p$ splits. Multiplicativity
therefore gives
\begin{equation}\label{eq:aK}
    a_K(m)\leq \tau(m)
\end{equation}
(see, e.g., \cite[Ch.~I, \S3]{Neukirch}).

If $D\equiv 2,3\pmod 4$, then $R=\OO_K$, so
\begin{equation}\label{eq:aR_bound_tau}
    a_R(m)\leq a_K(m)\leq \tau(m).
\end{equation}

If $D\equiv 1\pmod 4$, then
\begin{equation}
    R\subset \OO_K=\Z\!\left[\frac{1+\sqrt D}{2}\right],
    \qquad [\OO_K:R]=2,
\end{equation}
and $2\OO_K\subset R$. For an $R$-ideal $I$ of index $m$, set
$J=I\OO_K$. Then
\begin{equation}
    2J\subset I\subset J.
\end{equation}
Moreover
\begin{equation}
    2m=[\OO_K:R][R:I]=[\OO_K:I]=[\OO_K:J]\,[J:I].
\end{equation}
Since $J/I$ is a quotient of the two-dimensional $\F_2$-vector space
$J/2J$, we have 
\begin{equation}
    [J:I]\in\{1,2,4\}.
\end{equation} 
The case $[J:I]=4$ is not possible, as this would force $I=2J$, contradicting $I\OO_K=J$. Hence 
\begin{equation}
    [\OO_K: J]\in\{2m,m\}.
\end{equation}
For fixed $J$, if $[\OO_K: J]=2m$ then $[J:I]=1$ so that $I=J$; and if $[\OO_K:J]=m$ then $[J:I]=2$ so that $I$ is one of the three one-dimensional subspaces of $J/2J$.
Therefore
\begin{equation}\label{eq:aR_bound_5tau}
    a_R(m)\leq a_K(2m)+3a_K(m)
    \leq \tau(2m)+3\tau(m)
    \leq 5\tau(m).
\end{equation}
Combining the bounds \eqref{eq:aR_bound_tau}, \eqref{eq:aR_bound_5tau}, and \eqref{eq:F-ideal} gives the bound \eqref{eq:divisor_bound}. The
bound \eqref{eq:divisor_bound_epsilon} then follows from the standard estimate
$\tau(m)\ll_\varepsilon m^\varepsilon$ (see, e.g., \cite[Theorem~315]{HardyWright}).
\end{proof}

\begin{lemma}[Solutions within a family]\label{lem:family-count}
Let $D\geq 2$ be a nonsquare integer, let $n\neq 0$, and let $T\geq 2$.
The number of integer solutions $(X,Y)$ of $X^2-DY^2=n$ lying in a
single Pell family and satisfying $|Y|\leq T$ is $O(\log T)$, with an
absolute implied constant.
\end{lemma}

\begin{proof}
        For $\alpha=X+Y\sqrt D\in R$ we write
    $\overline{\alpha}=X-Y\sqrt D$ for its Galois conjugate and
    $\Norm(\alpha)=\alpha\overline{\alpha}=X^2-DY^2$ for its norm. Solutions $(X,Y)$ to $X^2-DY^2=n$ are in bijection with elements $\alpha\in R$ with
    $\Norm(\alpha)=n$.
    By Definition~\ref{def:pell-family}, two solutions
    lie in the same Pell family precisely when the associated elements differ
    by a unit of $R$.
    
    The units of $R$ of norm one form a subgroup 
    \begin{equation}
        R^{\times}_{+1}\defeq \{u\in R^\times : \Norm(u)=1\}
    \end{equation}
    of the unit group $R^\times$. Dirichlet's unit theorem  gives $R^\times=\{\pm1\}\times\langle
    \eta\rangle$ for a fundamental unit $\eta>1$; consequently
    $R^{\times}_{+1}=\{\pm\varepsilon_D^{\,k}:k\in\Z\}$, where $\varepsilon_D>1$
    is the smallest norm-one unit exceeding $1$
    (see, e.g., \cite[Ch.~I, \S7 and \S12]{Neukirch}). Write
    \begin{equation}\label{eq:eps-coords}
        \varepsilon_D=u_0+v_0\sqrt D,\qquad u_0,v_0\in\Z,\qquad
        u_0^2-Dv_0^2=1 .
    \end{equation}
    
    We claim $u_0,v_0\geq 1$. Since $\Norm(\varepsilon_D)=1$ we have
    $\overline{\varepsilon}_D=\varepsilon_D^{-1}$, and since $\varepsilon_D>1$, we have $\varepsilon_D^{-1}\in(0,1)$.
    Thus
    \begin{equation}
         2u_0=\varepsilon_D+\overline{\varepsilon}_D
          =\varepsilon_D+\varepsilon_D^{-1}>0
          \quad
          \text{ and }
          \quad
          2v_0\sqrt D=\varepsilon_D-\overline{\varepsilon}_D
               =\varepsilon_D-\varepsilon_D^{-1}>0.
    \end{equation}
    Therefore $u_0,v_0>0$, and hence $u_0,v_0\geq 1$ as they are integers. Thus
    \begin{equation}\label{eq:eps-lower}
        \varepsilon_D=u_0+v_0\sqrt D\geq 1+\sqrt2,
        \qquad\text{and so}\qquad
        \log\varepsilon_D\geq \log(1+\sqrt2)>0 .
    \end{equation}
    From
    $u_0^2=1+Dv_0^2\geq 1+D$, we have that
    \begin{equation}\label{eq:eps-vs-sqrtD}
        \varepsilon_D=u_0+v_0\sqrt D\geq u_0\geq \sqrt{1+D}>\sqrt D .
    \end{equation}
    
    Now fix a Pell family and pick a representative $\alpha_0=X_0+Y_0\sqrt D\in R$
    with $\Norm(\alpha_0)=n$. Every element of $R$ with norm $n$ that is
    $R^\times$-equivalent to $\alpha_0$ has the form $\pm\alpha_0\varepsilon_D^{\,k}$
    with $k\in\Z$, and replacing $\alpha_0$ by $-\alpha_0$ negates $(X,Y)$
    without changing $|Y|$. It therefore suffices to bound the number of
    $k\in\Z$ for which the solution corresponding to
    \begin{equation}\label{eq:family-param}
        \alpha_k\defeq \alpha_0\varepsilon_D^{\,k}=X_k+Y_k\sqrt D
    \end{equation}
    satisfies $|Y_k|\leq T$; the final count is at most twice this.
    
    We are free to replace $\alpha_0$ by $\alpha_0\varepsilon_D^{\,j}$ for any
    fixed $j\in\Z$, which merely reindexes \eqref{eq:family-param}. Since
    \begin{equation}
        \left|\frac{\alpha_0\varepsilon_D^j}{\overline{\alpha_0\varepsilon_D^j}}\right|=
        \left|\frac{\alpha_0}{\overline{\alpha_0}}\right|\left|\frac{\varepsilon_D^j}{\overline{\varepsilon_D^j}}\right|
        = \left|\frac{\alpha_0}{\overline{\alpha_0}}\right|\varepsilon_D^{2j},
    \end{equation}
    we may choose $j$ so that after reindexing
    \begin{equation}\label{eq:ratio}
        \varepsilon_D^{-2} < \left|\frac{\alpha_0}{\overline{\alpha_0}}\right|
        \le 1.
    \end{equation}
    Because $|\alpha_0\overline{\alpha_0}|=|\Norm(\alpha_0)|=|n|\geq 1$, the two
    inequalities in \eqref{eq:ratio} give, respectively,
    \begin{equation}\label{eq:alpha-lower}
        |\alpha_0|^2=|\alpha_0\overline{\alpha_0}|\cdot
            \left|\frac{\alpha_0}{\overline{\alpha_0}}\right|
        >\varepsilon_D^{-2}
        \quad \text{ and } \quad
        |\overline{\alpha_0}|^2
        =|\alpha_0\overline{\alpha_0}|\cdot
            \left|\frac{\overline{\alpha_0}}{\alpha_0}\right|
        \geq |n|\geq 1,
    \end{equation}
    so that
    \begin{equation}\label{eq:alpha-lower-both}
        |\alpha_0|> \varepsilon_D^{-1}
        \qquad\text{and}\qquad
        |\overline{\alpha_0}|\ge 1 .
    \end{equation}

    Define
    \begin{equation}
        A_k\defeq \alpha_0\varepsilon_D^{\,k},
        \quad 
        \text{ and }
        \quad
        B_k\defeq \overline{\alpha_0}\varepsilon_D^{-k},
    \end{equation}
    so that $\alpha_k=A_k$ and $\overline{\alpha_k}=B_k$ (using
    $\overline{\varepsilon}_D=\varepsilon_D^{-1}$). Subtracting conjugates in equation
    \eqref{eq:family-param},
    \begin{equation}\label{eq:Yk}
        Y_k=\frac{\alpha_k-\overline{\alpha_k}}{2\sqrt D}
           =\frac{A_k-B_k}{2\sqrt D}.
    \end{equation}
    By the inequality \eqref{eq:ratio},
    $|B_0|/|A_0|=|\overline{\alpha_0}/\alpha_0|\geq 1$, and hence, for $k\geq 1$,
    \begin{equation}\label{eq:ratio-k}
        \frac{|A_k|}{|B_k|}
        =\frac{|A_0|}{|B_0|}\,\varepsilon_D^{2k}
        \geq \varepsilon_D^{2k}\cdot\varepsilon_D^{-2}
        =\varepsilon_D^{2(k-1)} .
    \end{equation}
    In particular, by the first inequality in equation \eqref{eq:eps-lower}, for $k\geq 2$ we have 
    \begin{equation}
        \frac{|A_k|}{|B_k|}\geq
        \varepsilon_D^{2}\geq(1+\sqrt2)^2>2,
    \end{equation} 
    so that
    \begin{equation}
        |A_k-B_k|\geq |A_k|-|B_k|\geq \tfrac12|A_k|
    \end{equation}
     Combining this with
    \eqref{eq:Yk}, the lower bound $|\alpha_0|\geq\varepsilon_D^{-1}$ from
    \eqref{eq:alpha-lower-both}, and $\sqrt D<\varepsilon_D$ from
    \eqref{eq:eps-vs-sqrtD}, we obtain, for $k\geq 2$,
    \begin{equation}\label{eq:Yk-lower}
        |Y_k|
        \geq \frac{|A_k|}{4\sqrt D}
        =\frac{|\alpha_0|\,\varepsilon_D^{\,k}}{4\sqrt D}
        \geq \frac{\varepsilon_D^{-1}\varepsilon_D^{\,k}}{4\,\varepsilon_D}
        =\frac{\varepsilon_D^{\,k-2}}{4}.
    \end{equation}
    The analogous computation with the roles of $A_k$ and $B_k$ interchanged
    (using $|\overline{\alpha_0}|\geq 1\geq\varepsilon_D^{-1}$ from
    \eqref{eq:alpha-lower-both}) gives, for $k\leq -2$,
    \begin{equation}\label{eq:Yk-lower-neg}
        |Y_k|\geq \frac{\varepsilon_D^{\,|k|-2}}{4}.
    \end{equation}

    If $|Y_k|\leq T$ then equations \eqref{eq:Yk-lower} and \eqref{eq:Yk-lower-neg} force
    $\varepsilon_D^{\,|k|-2}\leq 4T$. Therefore, using the second inequality in equation \eqref{eq:eps-lower}, we have
    \begin{equation}\label{eq:k-bound}
        |k|\leq 2+\frac{\log(4T)}{\log\varepsilon_D}
        \leq 2+\frac{\log(4T)}{\log(1+\sqrt2)}.
    \end{equation}
     Hence the number of indices $k\in\Z$ with $|k|\geq 2$ and $|Y_k|\leq T$ is
    $O(\log T)$, with an absolute implied constant. Adjoining the three indices $k\in\{-1,0,1\}$ and the factor of $2$ from the choice of sign discussed in the sentence containing equation \eqref{eq:family-param} changes this only by an absolute constant. 
    Therefore the number of solutions in the family with $|Y|\leq T$ is $O(\log T)$, uniformly in $D$ and $n$.
\end{proof}

 \begin{lemma}[The perfect-square case]\label{lem:perfect-square}
If $D$ is a perfect square and $n\neq 0$, then
\begin{equation}
    \#\{(X,Y)\in\Z^2:X^2-DY^2=n\}\ll \tau(|n|).
\end{equation}
\end{lemma}

\begin{proof}
Write $D=d^2$. Then
\begin{equation}
    X^2-DY^2=(X-dY)(X+dY)=n.
\end{equation}
Each solution determines an ordered factorization
$n=n_1n_2$ with $n_1=X-dY$ and $n_2=X+dY$. Conversely, each ordered
factorization determines at most one pair $(X,Y)$, namely
\begin{equation}
    X=\frac{n_1+n_2}{2},\qquad
    Y=\frac{n_2-n_1}{2d},
\end{equation}
if these quantities are integers.
The number of ordered factorizations of $n$ is
\begin{equation}
    \#\{(n_1,n_2)\in \Z^2 : n=n_1 n_2\}=\#\{n_1\in \Z : n_1|n\}=2\tau(|n|)\ll \tau(|n|).
\end{equation}
\end{proof}

\section{The upper bound}\label{sec:upper-proof}

Motivated by Proposition~\ref{prop:curve}, we define
\begin{equation}\label{eq:Zsr}
    Z_{s,r}(N)\defeq
    \#\{(a,b)\in[1,N]^2:s f(b)=r f(a)\},
\end{equation}
and note that $M_{s,r}(N)\leq Z_{s,r}(N)$.

\begin{lemma}\label{lem:Zsr-bound}
For all integers $s>r\geq 1$ and $N\geq 1$,
\begin{equation}
    Z_{s,r}(N)
    \ll_f
    \tau\bigl(r(s-r)|\Delta|\bigr)\bigl(1+\log(sN)\bigr).
\end{equation}
In particular, $M_{s,r}(N)$ satisfies the same bound.
\end{lemma}

\begin{proof}
As before, the change of variables \eqref{eq:pell-change} gives an injective map from pairs $(a,b)\in[1,N]^2$ satisfying $s f(b)=r f(a)$ to solutions of 
\begin{equation}
    X^2-DY^2=n,
    \qquad D=sr,
    \qquad n=r(r-s)\Delta,
\end{equation}
with
\begin{equation}
    |Y|\leq 2AN+|B|\defeq T.
\end{equation}
Here
\begin{equation}
    |n|=r(s-r)|\Delta|.
\end{equation}

If $D$ is a perfect square, Lemma~\ref{lem:perfect-square} gives
\begin{equation}
    Z_{s,r}(N)\ll \tau(|n|),
\end{equation}
which is stronger than the claimed bound.

Suppose that $D$ is not a perfect square. Write $D=\ell^2D_0$, where
$D_0\geq 2$ is squarefree. The substitution $Z=\ell Y$ maps the solutions
injectively to solutions of
\begin{equation}
    X^2-D_0Z^2=n
\end{equation}
with $|Z|\leq \ell T$. By Lemma~\ref{lem:rep-bound}, these solutions lie
in at most $5\tau(|n|)$ Pell families, and by
Lemma~\ref{lem:family-count} each family contributes
$O(1+\log(\ell T))$ solutions. Since $\ell\leq\sqrt D\leq s$ and
$T\ll_f N$,
\begin{equation}
    1+\log(\ell T)\ll_f 1+\log(sN).
\end{equation}
Combining these counts yields the bound
\begin{equation}
    Z_{s,r}(N)\ll_f \tau(|n|)(1+\log(sN))
    = \tau(r(s-r)|\Delta|)(1+\log(sN)).
\end{equation}
\end{proof}

For $s\in\Z_{>0}$, let
\begin{equation}
    \sigma_{-1}(s)\defeq\sum_{d\mid s}\frac{1}{d}
\end{equation}
denote the sum of reciprocal divisors of $s$.

\begin{lemma}\label{lem:divisor-sums}
\begin{romenum}
\item For every integer $s\geq 2$,
\begin{equation}
    \sum_{r=1}^{s-1}\tau(r)\tau(s-r)
    \ll s\,\sigma_{-1}(s)(\log 2s)^2.
\end{equation}
\item For every $X\geq 2$,
\begin{equation}
    \sum_{2\leq s\leq X}
    \frac{\sigma_{-1}(s)(\log 2s)^2}{s}
    \ll (\log 2X)^3.
\end{equation}
\end{romenum}
\end{lemma}

\begin{proof}
For part (i), since the divisors of $k$ pair as $d$ and $k/d$,
\begin{equation}\label{eq:tau-bound}
    \tau(k)\leq 2\#\{d\mid k:d\leq\sqrt{k}\}.
\end{equation}
Defining 
\begin{equation}
    N_{d,e}(s)\defeq\#\{0<r<s:d\mid r,\ e\mid s-r\},
\end{equation}
the bound \eqref{eq:tau-bound} implies
\begin{equation}\label{eq:sum_tau(r)tau(s-r)_basic_bound}
    \sum_{r=1}^{s-1}\tau(r)\tau(s-r)
    \leq
    4\sum_{d\leq\sqrt{s}}\sum_{e\leq\sqrt{s}}N_{d,e}(s).
\end{equation}
The congruences \(r \equiv 0 \pmod d\) and \(s \equiv r \pmod e\) are solvable only if $\gamma\defeq\gcd(d,e)$ divides $s$,
and in that case the solutions form an arithmetic progression modulo
$de/\gamma$. Thus
\begin{equation}
    N_{d,e}(s)\leq \frac{s\gamma}{de}+1.
\end{equation}
The terms $1$ contribute $O(s)$. For the remaining terms, write
$d=\gamma d'$ and $e=\gamma e'$ to obtain
\begin{salign}
    \sum_{\substack{d,e\leq\sqrt{s}\\ \gcd(d,e)\mid s}}
    \frac{s\gcd(d,e)}{de}
    &\leq
    \sum_{\substack{\gamma\mid s\\ \gamma\leq\sqrt{s}}}
    \frac{s}{\gamma}
    \left(\sum_{d'\leq\sqrt{s}/\gamma}\frac{1}{d'}\right)
    \left(\sum_{e'\leq\sqrt{s}/\gamma}\frac{1}{e'}\right)\\
    &\ll s\,\sigma_{-1}(s)(\log 2s)^2.
\end{salign}
This bound, combined with the bound \eqref{eq:sum_tau(r)tau(s-r)_basic_bound}, proves part (i).

For part (ii), write $s=uv$:
\begin{salign}
    \sum_{2\leq s\leq X}
    \frac{\sigma_{-1}(s)(\log 2s)^2}{s}
    &\leq
    (\log 2X)^2
    \sum_{u\leq X}\frac{1}{u}
    \sum_{v\leq X/u}\frac{1}{uv}\\
    &\ll
    (\log 2X)^2(1+\log X)
    \sum_{u\geq 1}\frac{1}{u^2}\\
    &\ll (\log 2X)^3.
\end{salign}
\end{proof}

\begin{proof}[Proof of Theorem~\ref{thm:upper}]
By the definition of $S_F(N)$ (equation \eqref{eq:S-rearranged}), the bound $M_{s,r}(N)\leq Z_{s,r}(N)$, and
Lemma~\ref{lem:Zsr-bound}, we obtain the bound
\begin{salign}\label{eq:S_F_bound}
    S_F(N)
    &\ll_f
    \sum_{2\leq s\leq f(N)}\frac{1}{s^m}
    \sum_{r=1}^{s-1}
    \tau\bigl(r(s-r)|\Delta|\bigr)\bigl(1+\log(sN)\bigr).
\end{salign}
For $2\leq s\leq f(N)\ll_f N^2$ and $N\geq 2$, we have the bound
\begin{equation}\label{eq:1+log(sN)_bound}
    1+\log(sN)\ll_f\log N.
\end{equation}
 By submultiplicativity of the divisor function $\tau$,
\begin{equation}\label{eq:submultiplicative_bound}
    \tau\bigl(r(s-r)|\Delta|\bigr)
    \leq \tau(|\Delta|)\tau(r)\tau(s-r)
    \ll_f \tau(r)\tau(s-r).
\end{equation}
The bounds \eqref{eq:S_F_bound}, \eqref{eq:1+log(sN)_bound}, and
\eqref{eq:submultiplicative_bound}, together with Lemma~\ref{lem:divisor-sums}(i), give
\begin{equation}
    S_F(N)
    \ll_f
    \log N
    \sum_{2\leq s\leq f(N)}
    \frac{\sigma_{-1}(s)(\log 2s)^2}{s^{m-1}}.
\end{equation}

If $m\geq 3$, then $\sigma_{-1}(s)\leq 1+\log s$ and
\begin{equation}
    \sum_{s\geq 2}
    \frac{\sigma_{-1}(s)(\log 2s)^2}{s^{m-1}}<\infty,
\end{equation}
so $S_F(N)\ll_F\log N$.

If $m=2$, Lemma~\ref{lem:divisor-sums}(ii), with
$X=\max(f(N),2)\ll_f N^2$, gives
\begin{equation}
    S_F(N)\ll_F(\log N)^4.
\end{equation}
\end{proof}

\begin{proof}[Proof of Theorem~\ref{thm:visibility}, upper bounds]
By the upper bound \eqref{eq:invisible-S-intro} and Theorem~\ref{thm:upper},
\begin{equation}
    \#\mathrm{Invisible}_F(N)
    \ll_F
    \begin{cases}
        N\log N,& \text{ if } m\geq 3,\\
        N(\log N)^4,& \text{ if } m=2.
    \end{cases}
\end{equation}
\end{proof}

\section{The lower bound}\label{sec:lower-proof}

We now prove Theorem~\ref{thm:lower}. The proof has three steps. First we
construct one relation
\begin{equation}
    s_0f(b_0)=r_0f(a_0)
\end{equation}
with $s_0r_0$ nonsquare. Then we propagate this relation along a Pell
family. Finally, we count the lattice points blocked by these relations.

\begin{lemma}[A seed pair]\label{lem:seed}
There exist integers $a_0>b_0>n_f$ and coprime integers $s_0>r_0\geq 1$
such that
\begin{equation}\label{eq:seed-relation}
    s_0f(b_0)=r_0f(a_0)
\end{equation}
and $s_0r_0$ is not a perfect square.
\end{lemma}

\begin{proof}
Choose $b_0>n_f$ such that
\begin{equation}
   f(b_0)\geq 2.
\end{equation}
For $k\geq 1$, define
\begin{equation}
    a_k\defeq b_0+k f(b_0)
    \qquad 
    \text{ and }
    \qquad
    s_k\defeq\frac{f(a_k)}{f(b_0)}.
\end{equation}
Since $a_k\equiv b_0\pmod{f(b_0)}$ and $f\in\Z[x]$, we have
$f(a_k)\equiv f(b_0) \equiv 0\pmod{f(b_0)}$. Thus $s_k$ is a positive
integer and
\begin{equation}
    \gcd(f(a_k),f(b_0))=f(b_0).
\end{equation}
Moreover, expanding $f\left(b_0+kf(b_0)\right)$ gives the explicit formula
\begin{equation}
    s_k=A f(b_0)k^2+(2Ab_0+B)k+1.
\end{equation}
Since $a_k>b_0>n_f$, the monotonicity condition \eqref{eq:monotone}
gives $s_k>1$.

Call $k$ bad if $s_k=y^2$ for some integer $y\geq 1$.
Define
\begin{equation}
    U_k\defeq 2Aa_k+B=2A(b_0+kf(b_0))+B.
\end{equation}
For a bad $k$,
$f(a_k)=f(b_0)y^2$, and completing the square gives
\begin{equation}\label{eq:seed-pell}
    U_k^2-4Af(b_0) y^2=\Delta.
\end{equation}
We now distinguish according to whether $4Af(b_0)$ is a square.

Suppose that $4Af(b_0)=d^2$ for some integer $d\geq 1$. Then
\begin{equation}
    (U_k-dy)(U_k+dy)=\Delta.
\end{equation}
There are only $2\tau(|\Delta|)$ ordered integer factorizations of the
fixed nonzero integer $\Delta$. Each such factorization determines at
most one pair $(U_k,y)$, and $U_k$ determines $k$. Hence there are
$O_f(1)$ bad values of $k$ in this case.

Suppose instead that $4Af(b_0)$ is not a perfect square. Write
\begin{equation}
    4Af(b_0)=\ell^2D_0,
\end{equation}
where $\ell$ and $D_0$ are positive integers with $D_0\geq 2$ squarefree. Defining $Z\defeq \ell y$, equation
\eqref{eq:seed-pell} becomes
\begin{equation}
    U_k^2-D_0Z^2=\Delta.
\end{equation}
If $k\leq T$ is bad, then $|U_k|\ll_f T$, and the equation implies
$|Z|\ll_f T$. Since $U_k$ determines $k$, Lemma~\ref{lem:rep-bound}
and Lemma~\ref{lem:family-count} show that the number of bad $k\leq T$
is $O_f(\log T)$.

In either case, the number of bad $k\leq T$ is $o(T)$. We may therefore
choose a nonbad integer $k_0\geq 1$. Define
\begin{equation}
    a_0\defeq a_{k_0},
    \qquad s_0\defeq s_{k_0},
    \qquad r_0\defeq 1.
\end{equation}
Then $a_0>b_0>n_f$, the integers $s_0>r_0\geq 1$ are coprime,
the equality \eqref{eq:seed-relation} holds, and $s_0r_0=s_0$ is not a perfect
square.
\end{proof}

\begin{proposition}[Propagation along a Pell family]\label{prop:propagation}
There exist coprime integers $s_0>r_0\geq 1$, an integer $j_0\geq 0$, and
pairs of integers $(a_j,b_j)_{j\geq j_0}$ such that:
\begin{romenum}
\item $n_f<b_j<a_j$ for all $j\geq j_0$, and $a_j$ is strictly increasing;
\item $s_0f(b_j)=r_0f(a_j)$ for all $j\geq j_0$, and consequently
$s(a_j,b_j)=s_0$ and $r(a_j,b_j)=r_0$;
\item $\log a_j\asymp_f j$ as $j\to\infty$.
\end{romenum}
In particular,
\begin{equation}\label{eq:Msr-lower}
    M_{s_0,r_0}(N)\gg_f\log N.
\end{equation}
\end{proposition}

\begin{proof}
Let $(a_0,b_0,s_0,r_0)$ be the seed data from Lemma~\ref{lem:seed}, and
put
\begin{equation}
    D\defeq s_0r_0,
    \qquad
    n\defeq r_0(r_0-s_0)\Delta.
\end{equation}
The change of variables \eqref{eq:pell-change} sends the seed pair $(a_0,b_0)$ to
\begin{equation}
    X_0 \defeq r_0(2Aa_0+B),
    \qquad
    Y_0 \defeq 2Ab_0+B,
\end{equation}
a positive solution of the generalized Pell equation
\begin{equation}
    X^2-DY^2=n.
\end{equation}
Define
\begin{equation}
    \alpha_0\defeq X_0+Y_0\sqrt D,
    \qquad
    R\defeq\Z[\sqrt D],
    \qquad
    M\defeq 2Ar_0.
\end{equation}
Let $\varepsilon_D>1$ be the fundamental norm-one unit of $R$. Since
the quotient ring $R/MR$ is finite, there is an integer $L\geq 1$ such that
\begin{equation}\label{eq:unit-congruence}
    \varepsilon_D^L\equiv 1\pmod{MR}.
\end{equation}
For $j\geq 0$, define $X_j,Y_j\in\Z$ by
\begin{equation}
    X_j+Y_j\sqrt D\defeq\alpha_0\varepsilon_D^{jL}.
\end{equation}
Then $X_j^2-DY_j^2=n$, and the congruence \eqref{eq:unit-congruence} gives
\begin{equation}
    X_j\equiv X_0\pmod{2Ar_0},
    \qquad
    Y_j\equiv Y_0\pmod{2Ar_0}.
\end{equation}
Combining these congruences with the definitions of $X_0$ and $Y_0$ shows that
\begin{equation}
    a_j\defeq\frac{X_j/r_0-B}{2A},
    \qquad
    b_j\defeq\frac{Y_j-B}{2A}
\end{equation}
are integers.

Writing $\overline{\alpha_0}=X_0-Y_0\sqrt D$, we have
\begin{equation}
    X_j=\frac{\alpha_0\varepsilon_D^{jL}
        +\overline{\alpha_0}\varepsilon_D^{-jL}}{2},
    \qquad
    Y_j=\frac{\alpha_0\varepsilon_D^{jL}
        -\overline{\alpha_0}\varepsilon_D^{-jL}}{2\sqrt D}.
\end{equation}
It follows that $X_j$ and $Y_j$ are positive and strictly increasing for all sufficiently large $j$ and that $X_j/Y_j\to\sqrt D$ as $j\to\infty$. Moreover,
\begin{equation}
    a_j = \frac{X_j/r_0-B}{2A}
    = \frac{\alpha_0}{4Ar_0}\varepsilon_D^{jL}
      \left(1+\frac{\overline{\alpha_0}}{\alpha_0}
      \varepsilon_D^{-2jL}\right)-\frac{B}{2A}.
\end{equation}
Consequently,
\begin{equation}
    \log a_j
    =jL\log(\varepsilon_D)
     +\log\left(\frac{\alpha_0}{4Ar_0}\right)
     +O_f\left(\varepsilon_D^{-jL}\right)
    =jL\log\varepsilon_D+O_f(1)\asymp_f j.
\end{equation}
Moreover,
\begin{equation}
    \lim_{j\to\infty} \frac{a_j}{b_j}
    = \lim_{j\to\infty} \frac{X_j/r_0-B}{Y_j-B}
    =\frac{\sqrt D}{r_0} 
    =\sqrt{\frac{s_0}{r_0}}>1.
\end{equation}
Hence, we may choose $j_0$ so that $n_f<b_j<a_j$ and $a_j$ are strictly increasing for every $j\geq j_0$.

Substituting
$X_j=r_0(2Aa_j+B)$ and $Y_j=2Ab_j+B$ into
$X_j^2-DY_j^2=n$ and using
$(2Ax+B)^2=4Af(x)+\Delta$ gives
\begin{equation}
    s_0f(b_j)=r_0f(a_j).
\end{equation}
Since $\gcd(s_0,r_0)=1$ and $f(a_j),f(b_j)>0$, we have $s(a_j,b_j)=s_0$ and $r(a_j,b_j)=r_0$. The growth estimate for
$a_j$ now gives \eqref{eq:Msr-lower}.
\end{proof}

\begin{lemma}[Blocked heights]\label{lem:blocked-heights}
Let $m\geq 1$ and $F=f^m$, and suppose $n_f<b<a\leq N$. If $h\leq N$ is a positive integer divisible by $s(a,b)^m$, then the point
$(a,h)$ is invisible along $F$.
\end{lemma}

\begin{proof}
The unique rational number $t$ satisfying $h=tF(a)$ is $t=h/F(a)>0$.
By the definitions of $s(a,b)$ and $r(a,b)$ given in equation \eqref{eq:g-s-r}, we have
$f(a)=s(a,b)\, g(a,b)$ and $f(b)=r(a,b)\, g(a,b)$.
Therefore
\begin{equation}
    tF(b)=h\frac{f(b)^m}{f(a)^m}
    =\frac{h}{s(a,b)^m}r(a,b)^m,
\end{equation}
is a positive integer. Since $b<a$, the integer $a$ is not the smallest positive integer $u$ for which $tF(u)$ is a positive integer.
Thus $(a,h)$ is invisible along $F$.
\end{proof}

\begin{proof}[Proof of Theorem~\ref{thm:lower}]
Let $s_0,r_0$ and $(a_j,b_j)_{j\geq j_0}$ be as in
Proposition~\ref{prop:propagation}, and let $m\geq 1$. By
the sum \eqref{eq:S-rearranged} for $S_F(N)$ and  the lower bound \eqref{eq:Msr-lower} for $M_{s,r}(N)$, we have
\begin{equation}
    S_F(N)\geq\frac{M_{s_0,r_0}(N)}{s_0^m}\gg_F\log N.
\end{equation}
This, together with Theorem~\ref{thm:upper}, gives
$S_F(N)\asymp_F\log N$ when $m\geq 3$.

Let
\begin{equation}
    \mathcal{T}(N)\defeq\{j\geq j_0:a_j\leq N\}.
\end{equation}
By Proposition~\ref{prop:propagation}(iii), there is a constant $C_f>0$
such that $\log a_j\leq C_fj$ for all sufficiently large $j$. Hence all
sufficiently large integers $j\leq C_f^{-1}\log N$ belong to
$\mathcal{T}(N)$, and therefore
\begin{equation}
    \#\mathcal{T}(N)\gg_f\log N.
\end{equation}
For each $j\in\mathcal{T}(N)$,
Lemma~\ref{lem:blocked-heights} shows that every point
\begin{equation}
    (a_j,h),
    \qquad
    h\in\{s_0^m,2s_0^m,3s_0^m,\ldots\}\cap[1,N],
\end{equation}
is invisible along $F$. The first coordinates $a_j$ are distinct, so for
$N\geq 2s_0^m$,
\begin{equation}
    \#\mathrm{Invisible}_F(N)
    \geq
    \#\mathcal{T}(N)\left\lfloor\frac{N}{s_0^m}\right\rfloor
    \gg_F N\log N.
\end{equation}
\end{proof}

\begin{proof}[Proof of Theorem~\ref{thm:visibility}, lower bounds]
The lower bounds follow from Theorem~\ref{thm:lower}.
\end{proof}

\begin{remark}
For $m=2$, Theorems~\ref{thm:upper} and \ref{thm:lower} give
\begin{equation}
    \log N\ll_F S_F(N)\ll_F(\log N)^4
\end{equation}
and
\begin{equation}
    N\log N\ll_F\#\mathrm{Invisible}_F(N)\ll_F N(\log N)^4.
\end{equation}
\end{remark}

\end{document}